\documentclass{amsart}

\usepackage[inactive]{srcltx} 
\usepackage{amsmath, amsthm, amscd, amsfonts, amssymb, graphicx, color, extarrows}
 \newtheorem{thm}{Theorem}[section]
 \newtheorem{cor}[thm]{Corollary}
 \newtheorem{lem}[thm]{Lemma}
  
 \newtheorem{prop}[thm]{Proposition}
 \theoremstyle{definition}
 \newtheorem{defn}[thm]{Definition}
 \theoremstyle{remark}
 
 \newtheorem{exam}[thm]{Example}
 \numberwithin{equation}{section}
\newcommand{\T}{ \textbf{t}_{\lambda}}
\newcommand{\PT}{\mathfrak{U}\widehat{\otimes} \mathfrak{U}}

\newcommand{\uu}{\mathfrak{U}}
\begin{document}

\title[Symmetrically pseudo-amenable Banach algebras]
 {Symmetrically pseudo-amenable Banach algebras}

\author{Hoger Ghahramani and  Parvin Zamani}
\thanks{{\scriptsize
\hskip -0.4 true cm \emph{MSC(2020)}: 46H20; 46H25; 47B47;  46H99.
\newline \emph{Keywords}:symmetrically pseudo-amenable, pseudo-amenable, amenable, Jordan derivation, Lie derivation.  \\}}

\address{Department of Mathematics, Faculty of Science, University of Kurdistan, P.O. Box 416, Sanandaj, Kurdistan, Iran.}
\email{h.ghahramani@uok.ac.ir; hoger.ghahramani@yahoo.com}

\address{Department of Mathematics, Faculty of Science, University of Kurdistan, P.O. Box 416, Sanandaj, Kurdistan, Iran.}
\email{parvin.zamani88@gmail.com}


\address{}

\email{}

\thanks{}

\thanks{}

\subjclass{}

\keywords{}

\date{}

\dedicatory{}

\commby{}


\begin{abstract}
We introduce and study a new notion of amenability called symmetric pseudo-amenability. We obtain some properties of symmetrically pseudo-amenable Banach algebras and with examples, we compare this type of amenability with some other types of amenability. We also provide some special classes of symmetrically pseudo-amenable Banach algebras. Finally, Jordan derivations and Lie derivations from a class of Banach algebras into appropriate Banach bimodules are investigated using the notion of symmetric pseudo-amenability.
\end{abstract}

\maketitle
\section{Introduction }
B. E. Johnson studied cohomology of Banach algebras in \cite{johnson 1972} and defined the concept of amenable Banach algebra which was based on the amenability of locally compact groups and proved in \cite{johnson 1972 2} that a Banach algebra $\uu$ is amenable if and only if $\uu$ has a bounded approximate diagonal. In the following, many studies were conducted on amenability of Banach algebras, and various other types of amenability have been introduced and studied,  see \cite{dal, ghahramani 2004, khodakarami, runde} for a comprehensive survey of results of this type. In \cite{johnson 1996}, Johnson introduced symmetric amenable Banach algebras  as a special class of amenable Banach algebras. He called a Banach algebra $\uu$ is symmetrically amenable if it has a bounded approximate diagonal consisting of symmetric tensors, and then applied this concept to the study of Jordan derivations and Lie derivations. The study of Jordan derivations and Lie derivations also has a long history. The common problem regarding these mappings is when Jordan or Lie derivations can be characterized in terms of derivations? Many studies have been carried out in line with the this problem raised, and here we only refer to Johnson's results. Johnson showed in \cite[Theorem 6.2]{johnson 1996} that any continuous Jordan derivation from a symmetrically amenable Banach algebra $\uu$ into a Banach $\uu$-bimodule is a derivation. As a consequence he obtained the same result for continuous Jordan derivations on arbitrary $C^{*}$-algebras, although not every $C^{*}$-algebra is symmetrically amenable. Also, in \cite[Theorem 9.2]{johnson 1996} Johnson showed that any continuous Lie derivation from a symmetrically amenable Banach algebra $\uu$ into a Banach $\uu$-bimodule decomposed into the sum of a continuous derivation and a continuous center-valued trace, and as a consequence he obtained the same result for continuous Lie derivations on arbitrary $C^{*}$-algebras. To see the historical course and other results in this study path, we refer to \cite{alaminos0, alaminos, beh, cus, gh0, gh1, her, li, mar, math, mie, per, sinc} and the references therein. 
\par 
As a generalization of amenability, F. Ghahramani and Y. Zhang in \cite{ghahramani 2007} introduced and studied the notion of pseudo-amenability, which is
based on existence of an approximate diagonal for Banach algebras. Precisely, a Banach algebra $\uu$ is called pseudo-amenable if it has an approximate diagonal which is not necessarily bounded. According to this definition and with the idea of the definition of symmetric amenability, in the continuation of studies related to amenability, in this article we introduce the concept of symmetric pseudo-amenability. We say Banach algebra $\uu$ is \textit{symmetrically pseudo-amenable} if it has an approximate diagonal consisting of symmetric tensors which is not necessarily bounded. This concept is a generalization of symmetrically amenable Banach algebras, which is a special class of pseudo-amenable Banach algebras. In this article, we study symmetrically pseudo-amenable Banach algebras and identify some of their properties and compare this type of amenability with some other types of amenability with examples, and we also present some special classes of Banach algebras that are symmetrically pseudo-amenable. According to \cite[Theorem 6.2]{johnson 1996} and \cite[Theorem 9.2]{johnson 1996}, the question arises whether it is possible to obtain Johnson's results for Jordan derivations and Lie derivations of other suitable Banach algebras into suitable Banach modules? We give answers to this question using the concept of symmetric pseudo-amenability and obtain generalizations of \cite[Theorem 6.2]{johnson 1996} and \cite[Theorem 9.2]{johnson 1996}.
\par 
This article is organized as follows. In section 2, definitions, notions and required tools are introduced. Section 3 is devoted to the study of properties of symmetrically pseudo-amenable Banach algebras and examples. In section 4, we present some special classes of symmetrically pseudo-amenable Banach algebras. In sections 5 and 6, respectively, Jordan and Lie derivations from a class of Banach algebras into appropriate Banach bimodules are investigated using the concept of symmetric pseudo-amenability.
\section{Preliminaries}
In this section we fix the notation, and give some basic definitions and points which will be used in the next sections. Let $\uu$ be an algebra and $X$ be a $\uu$-bimodule. Note that $\mathcal{Z}_{\uu}(X)$ represents the center of $X$, which is defined as 
\[ \mathcal{Z}_{\uu}(X)= \lbrace x\in X \, \vert \, ax=xa \, \,  \text{for all}\, \, a\in \uu\rbrace .\] 
A mapping $f:\uu\rightarrow X$ is central if $f(\uu)\subseteq \mathcal{Z}_{\uu}(X)$.
\par 
Recall that a linear mapping $\delta : \uu \rightarrow X$ is called a derivation if 
\[ \delta(ab)=\delta(a)b+a\delta(b);\]
a Jordan derivation if 
\[ \delta(ab+ba)=\delta(a) b+a\delta(b)+\delta(b) a+b\delta(a) ;\]
and a Lie derivation if 
\[ \delta(ab-ba)=\delta(a)b+a\delta(b)-\delta(b)a-b\delta(a),\]
for all $a,b\in \uu$. A linear mapping $\tau: \uu\rightarrow X$ is a trace, if $\tau([a,b])=0$ for all $a,b\in \uu$, where $[a,b]=ab-ba$ (Lie product). Note that a Lie derivation is central if and only if it is a central trace. 
\par 
Assume that $\uu$ is a Banach algebra and $X$ is a Banach $\uu$-bimodule. The space of all bounded derivations from $\uu$ into $X$ is denoted by $Z^1(\uu,X)$. A derivation $\delta$ is called inner derivation, if there is $x\in X$ such that $\delta(a)=ax-xa$ for all $a\in \uu$. Each inner derivation is bounded and $N^1(\uu,X)$ is the space of all inner derivations from $\uu$ into $X$. The first cohomology group of $\uu$ with coefficient in $X$ is the quotient space $H^1(\uu,X)=Z^1(\uu,X)/N^1(\uu,X)$. We observe that $X^*$ is a Banach $\uu$-bimodule with the following module operations
\[\langle a f, x \rangle=\langle f,  xa \rangle,\quad \langle f a, x \rangle=\langle f, a x \rangle\]
for $a\in  \uu$, $ x\in  X$ and $ f\in X^*$. We call $X^*$ the dual bimodule of $X$. Recall that a Banach algebra $\uu$ is said to be amenable if for every Banach $\uu$-bimodule $X$, we have $H^1(\uu,X^*)=\lbrace 0\rbrace$. The notion of an amenable Banach algebra was introduced by Johnson in 1972 \cite{johnson 1972}. A net $\lbrace \T\rbrace_{\lambda\in\Lambda}$ in the projective tensor product $\PT$ is called approximate diagonal if satisfies the following two conditions
\begin{enumerate}
\item[$(1)$] $ a\T -\T a  \longrightarrow 0 $;
\item[$(2)$] $ \pi (\T)a \longrightarrow a $
\end{enumerate}
for all $a\in \uu$, where the operations on $\PT$ are defined through 
\[a(b\otimes c)=ab\otimes c, \quad (b\otimes c)a=b\otimes ca  \]
for all $a,b,c\in \uu$. Here and in the sequel $\pi$ always denotes the product morphism from $\PT$ into $\uu$, specified by $\pi(a\otimes b)=ab$ for all $a,b\in \uu$. In \cite{johnson 1972 2} (see also \cite[2.9.65]{dal}), it is proved that a Banach algebra $\uu$ is amenable if and only if $\uu$ has a bounded approximate diagonal. The flip map on $\PT$ is defined by 
\[ (a\otimes b)^{\circ}=b\otimes a \]
for $a,b \in \uu$ and an element $\textbf{t}\in \PT$ is called symmetric if $\textbf{t}^{\circ}=\textbf{t}$. Johnson in \cite{johnson 1996} is introduced symmetric amenability of Banach algebras. He called a Banach algebra $\uu$ is symmetrically amenable if it has a bounded approximate diagonal consisting of symmetric tensors. The opposite algebra $\uu^{\circ}$ is the Banach space $\uu$ with product $a\circ b = ba$. An approximate diagonal in $\PT$ for $\uu^{\circ}$ is a net $\lbrace \T\rbrace_{\lambda\in\Lambda}$ in $\PT$ if it meets the following two conditions
\begin{enumerate}
\item[$(1)^{\circ}$] $ a \circ \T -\T \circ a \rightarrow 0 $;
\item[$(2)^{\circ}$] $ a \pi ^{\circ}(\T) - a\rightarrow 0$
\end{enumerate}
for all $a\in \uu$, where
\[a\circ(b\otimes c)=b\otimes ac, \quad (b\otimes c)\circ a=ba\otimes c \quad  \text{and} \quad \pi^{\circ}(b\otimes c)=cb \]
for all $a,b,c\in \uu$. There are a number of obvious relationships between these operations, for example $a\circ \textbf{t}^{\circ}=(a\textbf{t})^{\circ}$ ($a\in\uu, \textbf{t}\in \PT$). The Banach algebra $\uu$ is symmetrically amenable if and only if there is a bounded net $\lbrace \T\rbrace_{\lambda\in\Lambda}$ in $\PT$ such that satisfies $(1)$, $(2)$, $(1)^{\circ}$ and $(2)^{\circ}$ (\cite[Proposition 2.2]{johnson 1996}). The properties and examples of symmetrically amenable Banach algebras can be found in \cite{johnson 1996}. 
\par 
In \cite{ghahramani 2004}, the authors have introduced and studied a generalization of amenability called approximate amenability, which is based on a property of
derivations from the algebra. Precisely, a Banach algebra $\uu$ is approximately amenable if, for every Banach $\uu$-bimodule $X$, every bounded derivation $d$ from $\uu$ into the dual bimodule $X^*$ is approximately inner, which means that there is a net $\lbrace x_{\lambda}\rbrace_{\lambda\in\Lambda}\subset X^*$ such that $d(a)=\lim _{\lambda}(ax_{\lambda}-x_{\lambda}a)$ for all $a\in \uu$. In \cite{ghahramani 2007} pseudo-amenability is presented as another generalization of amenability, which is based on approximate diagonals. Precisely, a Banach algebra $\uu$ is pseudo-amenable if it has an approximate diagonal which is not necessarily bounded. In \cite[Theorem 3.1]{ghahramani 2007} it is proved that a Banach algebra $\uu$ is approximately amenable if and only if the unitization $\uu^{\sharp}$ of $\uu$ is pseudo-amenable (also, see \cite{ghahramani 2008}). The properties and examples of these kinds of amenabilities can be found in \cite{ghahramani 2004, ghahramani 2007, ghahramani 2008}.
\section{Basic properties and comparisons}
In this section, we present definition and some basic properties of symmetric pseudo-amenability, and with examples, we compare this type of amenability with some other types. Throughout this section, $\uu$ is a Banach algebra. 
\begin{defn}
The Banach algebra $\uu$ is \textit{symmetrically pseudo-amenable} if it has an approximate diagonal consisting of symmetric elements. 
\end{defn}
In view of this definition, it is clear that a symmetrically amenable Banach algebras is symmetrically pseudo-amenable and a symmetrically pseudo-amenable Banach algebra is pseudo-amenable. In the next proposition, a condition equivalent to the symmetric pseudo-amenablity is presented, the proof of which is similar to \cite[Proposition 2.2]{johnson 1996}, and its proof is omitted.
\begin{prop}\label{equivalent symmetric}
$ \uu $ is symmetrically pseudo-amenable if and only if there exists a net $ \{\T\}_{\lambda\in\Lambda} $ in $ \PT $ which satisfies
\begin{enumerate}
\item[(i)] $ a \T -\T a  \longrightarrow 0 $;
\item[(ii)] $ \pi (\T)a \longrightarrow a $;
\item[(iii)]$a\circ\T-\T\circ a\longrightarrow 0 $;
\item[(iv)]$ a\pi^{\circ} (\T)\longrightarrow a$.
\end{enumerate}
for all $a\in\uu$.
\end{prop}
It should be noted that if ${\T} $ satisfies in conditions (i) to (iv) of the Proposition \ref{equivalent symmetric}, then for each $a\in\uu$, $ a\pi(\T)\longrightarrow a $ can be concluded from (i) and (ii), and $ \pi^{\circ}(\T)a\longrightarrow a $ can be concluded from (iii) and (iv).
 \par
We know that in order to show that the Banach algebra $\uu$ is pseudo-amenable, it is sufficient  to show that for each finite subset $F$ of $\uu$ and each $\varepsilon > 0$ there is an element $\textbf{t}$ of $\PT$ with $\Vert a\textbf{t}-\textbf{t}a \Vert <\varepsilon$ and $\Vert \pi(\textbf{t})a-a \Vert <\varepsilon$ for all $a\in F$. Now according to the definition of symmetric pseudo-amenability to show that $\uu$ is symmetrically pseudo-amenable, it is sufficient to show that for each finite subset $F$ of $\uu$ and each $\varepsilon > 0$ there is a symmetric element $\textbf{t}$ of $\PT$ such that $\Vert a\textbf{t}-\textbf{t}a \Vert <\varepsilon$ and $\Vert \pi(\textbf{t})a-a \Vert <\varepsilon$ for all $a\in F$. Also, according to Proposition \ref{equivalent symmetric}, in order to prove that $\uu$ is symmetrically pseudo-amenable, it is sufficient to show that for each finite subset $F$ of $\uu$ and each $\varepsilon > 0$ there is an element $\textbf{t}$ of $\PT$ with $\Vert a\textbf{t}-\textbf{t}a \Vert <\varepsilon$, $\Vert a\circ\textbf{t}-\textbf{t}\circ a \Vert <\varepsilon$, $\Vert \pi(\textbf{t})a-a \Vert <\varepsilon$ and $\Vert a \pi^{\circ}(\textbf{t})-a \Vert <\varepsilon$ for all $a\in F$.
\par 
As an application of Proposition \ref{equivalent symmetric}, we have the following result.
\begin{cor}\label{symmetric commutative}
If $\uu $ is a commutative pseudo-amenable Banach algebra, then $\uu $ is symmetrically pseudo-amenable.
\end{cor}
\begin{proof}
For commutative Banach algebras, in Proposition \ref{equivalent symmetric}, conditions (i) and (iii), and conditions (ii) and (iv)  are the same.
\end{proof}
In the previous corollary, we saw that pseudo-amenability is equivalent to symmetric pseudo-amenability on commutative Banach algebras. By using the next proposition, we can obtain an example of Banach algebras that are pseudo-amenable, but not symmetrically pseudo-amenable, and therefore, these classes of Banach algebras are different in general. The idea of the next result comes from \cite[Proposition 2.4]{johnson 1996}.
\begin{prop}\label{non symmetric}
Let $ \uu $ be a symmetrically pseudo-amenable Banach algebra, and $ z \neq 0$ be an element of $ \mathcal{Z}(\uu)$. Then there is a net $\{f_{\lambda}\}_{\lambda\in\Lambda} $ in $ \uu^{\ast} $ such that 
$ f_{\lambda}(ab-ba)\longrightarrow 0 $ for every $ a,b\in\uu $ and $ f_{\lambda}(z)\longrightarrow 1 $. Especially, if $ \uu $ is unital, then for $ z=1 $ there exists a net $ \{f_{\lambda}\}_{\lambda\in\Lambda} $ in $ \uu^{\ast} $ such that 
$ f_{\lambda}(ab-ba)\longrightarrow 0 $ for every $ a,b\in\uu $ and $ f_{\lambda}(1)\longrightarrow 1 $.
\end{prop}
\begin{proof}
Suppose that $ \{\T\}_{\lambda\in\Lambda} $ is a net in $\PT $ that satisfies to conditions (i) to (iv) of Proposition \ref{equivalent symmetric}, and 
$ g\in\uu^{*} $ is such that $ g(z)=1 $. For each $ \lambda\in\Lambda $ define
 \[f_{\lambda}(a)=g(\pi^{\circ}(a\T))\quad (a\in\uu).\]
Each $ f_{\lambda} $ is a bounded linear functional on $ \uu $. For every 
 $ \lambda\in\Lambda $ and $ a,b\in\uu $ we have 
 $ \pi^{\circ}(ab\T)=\pi^{\circ}(b\T a) $, and hence
  \begin{equation*}
\begin{split}
f_{\lambda}(ab-ba)&=g(\pi^{\circ}(ab\T -ba \T))\\&
=g(\pi^{\circ}(b\T a-ba\T))\\&
=g(\pi^{\circ}(b(\T a-a\T)))\longrightarrow 0
\end{split}
\end{equation*}
for each $ a,b\in\uu $. Also, since $ z\in\mathcal{Z}(\uu) $, we have
 \begin{equation*}
\pi^{\circ}(z\T)=\pi^{\circ}(\T\circ z)=\pi^{\circ}(\T)z\longrightarrow z
\end{equation*}
Now, it follows from $ g\in\uu^{*} $ that
\[ f_{\lambda}(z)\longrightarrow g(z)=1. \]
\end{proof}

In the next example, we present a Banach algebra, which is pseudo-amenable but not symmetrically pseudo-amenable.
\begin{exam}\label{not spa}
Let $ O_{n} $ be the Cuntz algebra, which is a unital amenable $ C^{*} $-algebra, and hence it is pseudo-amenable. There are members $ T_{1},...T_{n} $ in $ O_{n} $ such that $ T_{i}^{*}T_{i}=I $ and
 $ \sum_{i=1} ^{n}T_{i}T_{i}^{*}=I$. Suppose that $ O_{n} $ for $ n> 1 $ is symmetrically pseudo-amenable. According to Proposition \ref{non symmetric}, there exists a net $\{ f_{\lambda}\}_{\lambda\in\Lambda} $ of bounded linear functionals on $ O_{n} $ such that
 \[ f_{\lambda}(AB-BA)\longrightarrow0 \]
 for each $ A,B\in O_{n} $ and
 \[f_{\lambda}(I)\longrightarrow 1.\]
 Therefore
 \begin{equation*}
1=\lim_{\lambda}f_{\lambda}(I)=\lim_{\lambda}f_{\lambda}(\sum_{i=1}^{n}T_{i}T^{*}_{i})=\lim_{\lambda}\sum_{i=1}^{n}f_{\lambda}(T_{i}T^{*}_{i})
\end{equation*}  
On the other hand
\begin{equation*}
\lim_{\lambda}\sum_{i=1}^{n} f_{\lambda}(T_{i}T^{*}_{i})-\lim_{\lambda}\sum_{i=1}^{n}f_{\lambda}(T^{*}_{i}T_{i})
=\lim_{\lambda}\sum_{i=1}^{n}f_{\lambda}(T_{i}T^{*}_{i}-T^{*}_{i}T_{i})\longrightarrow 0
\end{equation*} 
So $ \lim_{\lambda}\sum_{i=1}^{n}f_{\lambda}(T^{*}_{i}T_{i})=1 $. but
\[ \lim_{\lambda}\sum_{i=1}^{n}f_{\lambda}(T^{*}_{i}T_{i})=\sum_{i=1}^{n}\lim_{\lambda}f_{\lambda}(T^{*}_{i}T_{i})=\sum_{i=1}^{n}\lim_{\lambda}f_{\lambda}(I)=n\]
So $ n=1 $, which contradicts $ n>1 $. Therefore, $ O_{n} $ for $ n> 1 $ cannot be a symmetrically pseudo-amenable Banach algebra.
\end{exam}
In fact, this example presents a Banach algebra that is amenable but not symmetrically pseudo-amenable. By using the next proposition, we can give an example of a symmetrically pseudo-amenable Banach algebra that is not amenable. To express this proposition, we first introduce some concepts.
\par 
Let $ \{X_{i}:i\in I\} $ be a collection of Banach spaces. Denote by $ \prod_{i\in I}X_{i} $ the product space of the collection. This is the space consisting of all mappings $ x:I\longrightarrow\bigcup_{i\in I} X_{i} $ for which $ x(i)\in X_{i} $, the linear operations being given coordinatewise. For $ 1\leq p<\infty$, we recall that the $\ell_{p}$-direct sum of the collection is
  \[\bigoplus_{i\in I}^{p}X_{i}=\bigg\{x\in\prod_{i\in I} X_{i}\, : \,\Vert x\Vert_{p}=\bigg (\sum_{i\in I}\Vert x(i)\Vert^{p}\bigg )^{\dfrac{1}{p}}<\infty \bigg\},\]
 and the $ c_{0} $-direct sum of the collection is
  \[\bigoplus_{i\in I}^{0}X_{i}=\bigg\{x\in\prod_{i\in I}X_{i}\, :\, \Vert x\Vert_{\infty}=max\Vert x(i)\Vert<\infty \, \text{and} \, \lim_{i}x(i)=0\bigg\}.\]
 For a collection $\lbrace \uu_{i}:i\in I\rbrace $ of Banach algebras, the sum $ \bigoplus_{i\in I}^{p}\uu_i$, $p\geqslant 1$ or $p=0 $, is also a Banach algebra with the multiplication being defined coordinatewise. If $ J\subset I $, then $ \bigoplus_{i\in J}^{p}\uu_{i} $ can be identified with the complemented closed ideal of $ \bigoplus_{i\in I} ^{p}\uu_{i}$ consisting of all $ x $ with $ x(i)=0$ for $ i\notin J $. We let $ P_{J} $ denote the associated projection from $ \bigoplus_{i\in I}^{p} \uu_{i} $ onto  $\bigoplus_{i\in J}^{p}\uu_{i} $. It should be noted that for $i_0\in I$, if $ \rho_{i_0}:\uu_{i_0}\longrightarrow\bigoplus_{i\in I}^{p}\uu_{i} $ is the natural embedding map, then 
  $\tilde{\rho}_{i_0}:\uu_{i_0}\widehat{\otimes}\uu_{i_0}\longrightarrow (\bigoplus_{i\in I}^{p}\uu_{i})\widehat{\otimes}(\bigoplus_{i\in I}^{p}\uu_{i})  $
 is given by
 $ \tilde{\rho}_{i_0}(\sum_{k=1}^{\infty}a_{k}^{i_0}\otimes b^{i_0}_{k})=\sum_{k=1}^{\infty}\rho_{i_0}(a^{i_0}_{k})\otimes \rho_{i_0}(b^{i_0}_{k}) $, where
 $ \sum_{k=1}^{\infty}a_{k}^{i_0}\otimes b_{k}^{i_0} \in \uu_{i_0}\widehat{\otimes}\uu_{i_0} $ is a bounded linear embedding with $ \Vert \tilde{\rho}\Vert\leq 1 $. Now we have the following proposition, the idea of proof of which is taken from \cite[Proposition 2.1]{ghahramani 2007}.
 \begin{prop}\label{directsum}
Suppose that for each $i\in I$, $ \uu_{i}$ is a symmetrically pseudo-amenable Banach algebra. Then so is $ \bigoplus_{i\in I}^{p}\uu_{i}$ for $ p\geq 1$ or $ p=0 $.
\end{prop}
\begin{proof}
Let $ \uu= \bigoplus_{i\in I}^{p}\uu_{i} $. Given $\varepsilon>0 $ and a finite set $ F\subset\uu $, we can choose a finite set $ J_{F,\varepsilon}\subset I $ for which
 $ \Vert P_{J_{F,\varepsilon}}(a)-a\Vert<\dfrac{\varepsilon}{2}$ for each $a\in F $. Since each $ \uu_{i} $ is symmetrically pseudo-amenable, by Proposition \ref{equivalent symmetric}, for every $i\in J_{F,\varepsilon}$ there is $\textbf{t}_{i}\in\uu_{i}\widehat{\otimes}\uu_{i} $ such that
\[ \Vert P_{i}(a)\textbf{t}_{i}-\textbf{t}_{i}P_{i}(a)\Vert<\dfrac{\varepsilon}{\mid J_{F,\varepsilon}\mid}; \]
\[ \Vert\pi(\textbf{t}_{i})P_{i}(a)-P_{i}(a)\Vert<\dfrac{\varepsilon}{2\mid J_{F,\varepsilon}\mid};\]
 \[\Vert P_{i}(a)\circ \textbf{t}_{i}-\textbf{t}_{i}\circ P_{i}(a)\Vert<\dfrac{\varepsilon}{\mid J_{F,\varepsilon}\mid} \]
 and
 \[\Vert P_{i}(a)\pi^{\circ}(\textbf{t}_{i})-P_{i}(a)\Vert<\dfrac{\varepsilon}{2\mid J_{F,\varepsilon}\mid} \]
 for each $a\in F$, where $ \mid J_{F,\varepsilon}\mid=\textit{card}J_{F,\epsilon} $ and $ P_{i} $ denotes the projection $ P_{\{i\}} $. Now we consider the embedding $\tilde{\rho}_{i}:\uu_{i}\widehat{\otimes}\uu_{i}\longrightarrow \uu\widehat{\otimes}\uu$ and choose the element $ \textbf{t}_{F,\varepsilon} $ in $ \PT $ as follows
 \[\textbf{t}_{F,\varepsilon}=\sum_{i\in J_{F,\varepsilon}}\tilde{\rho}_{i}(\textbf{t}_{i}). \]
For every $ a\in\uu $ we have $ a=\sum_{i\in I}P_{i}(a) $, and hence, according to the definition of $ \textbf{t}_{F,\varepsilon} $ and $ \tilde{\rho}_{i} $ we have
 \[a\textbf{t}_{F,\varepsilon}=\sum_{i\in J_{F,\varepsilon}}P_{i}(a)\tilde{\rho}_{i}(\textbf{t}_{i})=\sum_{i\in J_{F,\varepsilon}}\tilde{\rho}_{i}(P_{i}(a)\textbf{t}_{i}).\]
In the same way
 \[\textbf{t}_{F,\varepsilon}a=\sum_{i\in J_{F,\varepsilon}}\tilde{\rho}_{i}(\textbf{t}_{i}P_{i}(a)); \]
 \[a\circ \textbf{t}_{F,\varepsilon}=\sum_{i\in J_{F,\epsilon}} \tilde{\rho}_{i}(P_{i}(a)\circ \textbf{t}_{i}) \]
and
\[\textbf{t}_{F,\varepsilon}\circ a=\sum_{i\in J_{F,\varepsilon}} \tilde{\rho}_{i}(\textbf{t}_{i}\circ P_{i}(a)). \]
Therefore, for each $ a\in F $ we have
\begin{equation*}
\begin{split}
\Vert a\textbf{t}_{F,\varepsilon}-\textbf{t}_{F,\varepsilon}a\Vert &=\Vert\sum_{i\in J_{F,\varepsilon}}\tilde{\rho}_{i}(P_{i}(a)\textbf{t}_{i}-\textbf{t}_{i}P_{i}(a))\Vert\\&
\leq \sum_{i\in J_{F,\varepsilon}}\Vert\tilde{\rho}_{i}(P_{i}(a)\textbf{t}_{i}-\textbf{t}_{i}P_{i}(a))\Vert \\&
\leq\sum_{i\in J_{F,\varepsilon}}\Vert(P_{i}(a)\textbf{t}_{i}-\textbf{t}_{i}P_{i}(a)\Vert<\varepsilon
\end{split}
\end{equation*} 
 and
 \begin{equation*}
\begin{split}
\Vert \pi(\textbf{t}_{F,\varepsilon})a-a\Vert &=\Vert\pi(\textbf{t}_{F,\varepsilon}a)-a\Vert \\&
=\Vert\sum_{i\in J_{F,\varepsilon}}\pi(\tilde{\rho}_{i}(\textbf{t}_{i}P_{i}(a))-a\Vert\\&
=\Vert\sum_{i\in J_{F,\varepsilon}}\pi(\textbf{t}_{i})P_{i}(a)-a\parallel\\&
\leq\Vert\sum_{i\in J_{F,\varepsilon}}\pi(\textbf{t}_{i})P_{i}(a)-\sum_{i\in J_{F,\varepsilon}}P_{i}(a)\Vert+\Vert\sum_{i\in J_{F,\varepsilon}}P_{i}(a)-a\Vert\\&
\leq\sum_{i\in J_{F,\varepsilon}}\Vert\pi(\textbf{t}_{i})P_{i}(a)-P_{i}(a)\Vert + \Vert P_{J_{F,\varepsilon}}(a)-a\Vert<\dfrac{\varepsilon}{2}+\dfrac{\varepsilon}{2}=\varepsilon .
\end{split}
\end{equation*}
 Similarly, we have
 \[\Vert a\circ \textbf{t}_{F,\varepsilon}-\textbf{t}_{F,\varepsilon}\circ a\Vert<\varepsilon \]
 and
 \[\Vert a\pi^{\circ}(\textbf{t}_{F,\varepsilon})-a\Vert<\varepsilon \]
 for each $ a\in F $. So from Proposition \ref{equivalent symmetric} it follows that $\uu$ is symmetrically pseudo-amenable.
\end{proof}
 Now we are in a position to present the desired example.
 \begin{exam}\label{not amenable}
 For each $n\geqslant 1  $, let $ M_{n}(\mathbb{C}) $ be the algebra of $ n\times n $ matrices over the complex field $ \mathbb{C} $. Every $ M_{n}(\mathbb{C}) $ is symmetrically amenable because \[\textbf{t}=\dfrac{1}{n}\sum_{i,j=1}^{n}E_{ij}\otimes E_{ji},\] 
 where the $E_{ij}$ are the usual matrix units, is a symmetric diagonal in $ M_{n}(\mathbb{C}) $ (see \cite[Proposition 1.9.20]{dal}). According to Proposition \ref{directsum}, $ \bigoplus_{n\in \mathbb{N}}^{0}M_{n}(\mathbb{C}) $ is a symmetrically pseudo-amenable $ C^{*} $-algebra that is not amenable.
 \end{exam}
 We know that every symmetrically amenable Banach algebra is amenable and symmetrically pseudo-amenable. According to this point and the above examples, the following question arises:
 \\ \\
 \textbf{Question.} What is the relationship between the class of symmetrically amenable Banach algebras and the class of amenable and symmetrically pseudo-amenable Banach algebras? Is the Banach algebra that is  amenable and symmetrically pseudo-amenable, is the symmetrically amenable Banach algebra?
 \par 
 We continue this section by studying some hereditary properties.
  \begin{prop}\label{epimorphism}
 Let $ \uu $ and $ \mathcal{B} $ be two Banach algebras. If $ \uu $ is symmetrically pseudo-amenable and there is a continuous epimorphism $ \theta $ from $ \uu $ onto $ \mathcal{B} $, then $ \mathcal{B} $ is symmetrically pseudo-amenable. In particular, the quotient algebra $\uu / I$ is symmetrically pseudo-amenable for any two-sided closed ideal $ I $ of $\uu  $.
 \end{prop}
 \begin{proof}
 Consider the continuous linear mapping $ \theta\otimes\theta:\PT\longrightarrow  \mathcal{B}\widehat{\otimes}\mathcal{B}$. For each $ \textbf{t}\in\PT $ and $ a\in\uu $ we have the followings
 \[\theta\otimes\theta(a\textbf{t})=\theta(a)\theta\otimes\theta(\textbf{t})\quad \text{and} \quad  \theta\otimes\theta(\textbf{t}a)=\theta\otimes\theta(\textbf{t})a; \]
 \[\theta\otimes\theta(a\circ \textbf{t})=\theta(a)\circ\theta\otimes\theta(\textbf{t}) \quad \text{and} \quad \theta\otimes\theta(\textbf{t}\circ a)=\theta\otimes\theta(\textbf{t})\circ \theta(a);\]
 \[\pi(\theta\otimes\theta(\textbf{t}))=\theta(\pi(\textbf{t}))\quad \text{and} \quad \pi^{\circ}(\theta\otimes\theta(\textbf{t}))=\theta(\pi^{\circ}(\textbf{t})) .\]
 Considering these relationships and that $ \theta $ is epimorphism, it follows that $ \theta\otimes\theta $ maps any symmetric approximate diagonal for $ \uu $ to a symmetric approximate diagonal for $\mathcal{B}$.
 \end{proof}
 \begin{prop}\label{symmetric sa ideal}
 Let $\uu$ be a symmetrically pseudo-amenable Banach algebra, and let $ J $ be a two-sided closed ideal of $ \uu $. If $ J $ has an approximate identity $ \{e_{i}\}_{i\in I} $ such that the associated left and right multiplication operators  $ L_{i}:a\rightarrow e_{i}a $ and $ R_{i}:a\rightarrow ae_{i} $ from $ \uu $ into $ J $ are uniformly bounded, then $ J $ is symmetrically pseudo-amenable.
 \end{prop}
 \begin{proof}
 Under the condition on $ e_{i} $, there is a constant $ M\geq 1 $ such that $ \Vert e_{i} a \Vert\leq M\Vert a \Vert $ and $ \Vert ae_{i}\Vert\leq M\Vert a\Vert$
 for all $ e_{i} $ and all $ a\in\uu $. So $\Vert e_{i}\textbf{t}\Vert\leq M \Vert \textbf{t} \Vert $, $\Vert \textbf{t}e_{i}\Vert\leq M \Vert \textbf{t} \Vert $, $\Vert e_{i}\circ \textbf{t}\Vert\leq M\Vert \textbf{t}\Vert $ and $ \Vert \textbf{t}\circ e_{i}\Vert\leq M\Vert \textbf{t}\Vert $ for all $ e_{i} $ and all $ \textbf{t}\in\PT $.
\par 
To prove the theorem, we first prove the following claim.
\\
\textbf{Claim.} For each $ a\in J $ and $ \textbf{t}\in\PT $ we have
 \[e_{i}\circ (\textbf{t}a)\longrightarrow \textbf{t}a.\]
\textit{Reason.}  Suppose $ \textbf{t}=\sum_{j=1}^{\infty}a_{j}\otimes b_{j}\in \PT $. For each $ k\in \mathbb{N} $, we put $ \textbf{t}_{k}=\sum_{j=1}^{k}a_{j}\otimes b_{j} $. Since $ \textbf{t}_{k}\longrightarrow \textbf{t} $, it follows that for $ a\in J $, $ \textbf{t}_{k}a\longrightarrow \textbf{t}a $.  So for $ \varepsilon\geq 0 $, there exists a $k_{0}$ such that
 \[\Vert \textbf{t}_{k_{0}}a-\textbf{t}a\Vert<\dfrac{\varepsilon}{M+2}.\]
On the other hand, since $\{ e_{i}\}_{i\in I} $ is an approximate identity for $ J $ and $ \textbf{t}_{k_{0}}=\sum_{j=1}^{k_0}a_{j}\otimes b_{j} $, we have $ e_{i}\circ (\textbf{t}_{k_{0}}a)=\sum_{j=1}^{k_{0}}a_{j}\otimes e_{i}b_{j}a $ and therefore $ e_{i}\circ(\textbf{t}_{k_{0}}a)\rightarrow \textbf{t}_{k_{0}}a $. So there is an $ i_{0}\in I $ such that for every $ i\geqslant i_{0} $ we have
 \[\Vert e_{i}\circ(\textbf{t}_{k_{0}}a)-\textbf{t}_{k_{0}}a\Vert <\dfrac{\varepsilon}{M+2}.\]
Thus if $ i\geq i_{0} $, then 
\begin{equation*}
\begin{split}
\Vert e_{i}\circ(\textbf{t}a)-\textbf{t}a\Vert &\leq \parallel e_{i}\circ (\textbf{t}a)-e_{i}\circ(\textbf{t}_{k_{0}}a)\Vert +\Vert e_{i}\circ(\textbf{t}_{k_{0}}a)-\textbf{t}_{k_{0}}a\Vert +\Vert\textbf{t}_{k_{0}}a-\textbf{t}a\Vert\\&
\leq M\Vert\textbf{t}a-\textbf{t}_{k_{0}}a\Vert +\Vert e_{i}\circ(\textbf{t}_{k_{0}}a)-\textbf{t}_{k_{0}}a\Vert +\Vert \textbf{t}_{k_{0}}a-\textbf{t}a\Vert\\&
< M\dfrac{\varepsilon}{M+2}+\dfrac{\varepsilon}{M+2}+\dfrac{\varepsilon}{M+2}=\varepsilon .
\end{split}
\end{equation*}
Therefore, the claim is true.
\par 
 Now we prove the theorem.
\par 
Let $ \{\T\}_{\lambda\in \Lambda}\subseteq\PT $ be a symmetric approximate diagonal for $ \uu $. For an arbitrary $ \varepsilon> 0 $ and finite set $ F\subseteq J $, we choose $ \lambda\in\Lambda $ such that
 \[\Vert a\T-\T a\Vert M^{2}<\dfrac{\varepsilon}{2}\]
 and
 \[\Vert\pi(\T)a-a\Vert M<\frac{\varepsilon}{3}\] 
 for each $ a\in F $. Then, according to the proven claim and that $ \{e_{i}\}_{i\in I} $ is an approximate identity for $J$, we choose $i\in I$ so that
 \[\Vert ae_{i}-e_{i}a\Vert M\Vert\T\Vert<\dfrac{\varepsilon}{2};\]
 \[\Vert e_{i}a-a\Vert M\Vert\T\Vert<\dfrac{\varepsilon}{3}\]
 and
 \[\Vert e_{i}\circ(\T a)-\T a\Vert<\dfrac{\varepsilon}{3}\]
for $a\in F$.  We put
 \[\textbf{m}_{\lambda i}=(\T \circ e_{i})e_{i},\]
where $ \lambda\in\Lambda , i\in I $. We have the following relations
 \[a\textbf{m}_{\lambda i}-\textbf{m}_{\lambda i}a=[(a\T-\T a)\circ e_{i}]e_{i}+(\T\circ e_{i})(ae_{i}-e_{i}a) \]
 and
 \[ \pi(\textbf{m}_{\lambda i})a=\pi(e_{i}\circ\T)(e_{i}a-a)+\pi(e_{i}\circ(\T a)) \]
 for $ a\in J $. So, for $\lambda\in \Lambda$ and $i\in I$ chosen before, we have
\begin{equation*}
\begin{split}
\Vert a\textbf{m}_{\lambda i}-\textbf{m}_{\lambda i}a \Vert & \leq \Vert [(a\T-\T a)\circ e_{i}]e_{i}\Vert +\Vert (\T\circ e_{i})(ae_{i}-e_{i}a)\Vert\\&
\leq M^{2} \Vert a\T-\T a \Vert+M \Vert \T \Vert\Vert ae_{i}-e_{i}a \Vert\\&
< \dfrac{\varepsilon}{2}+\dfrac{\varepsilon}{2}=\varepsilon
\end{split}
\end{equation*}
and
\begin{equation*}
\begin{split}
\Vert\pi(\textbf{m}_{\lambda i})a-a \Vert &=\Vert \pi(e_{i}\circ\T)(e_{i}a-a) +\pi(e_{i}\circ(\T a))-a \Vert\\&
\leq \Vert\pi (e_{i}\circ\T)(e_{i}a-a) \Vert +\Vert\pi(e_{i}\circ(\T a))-\pi(\T a)\Vert +\Vert\pi(\T a)-a \Vert\\&
\leq \Vert e_{i}\circ\T \Vert \Vert e_{i}a-a \Vert +\Vert e_{i}\circ(\T a)-\T a \Vert +\Vert \pi(\T)a-a \Vert\\&
\leq M\Vert\T \Vert \Vert e_{i}a-a \Vert +\Vert e_{i}\circ(\T a)-\T a \Vert +\Vert \pi(\T)a-a \Vert\\&
<\dfrac{\varepsilon}{3}+\dfrac{\varepsilon}{3}+\dfrac{\varepsilon}{3}=\varepsilon
\end{split}
\end{equation*}
for each $ a\in F $. It is also checked routinely for each $(\lambda,i)\in\Lambda\times I $ that $ \textbf{m}^{\circ}_{\lambda i}=\textbf{m}_{\lambda i} $. Thus choosing an appropriate subnet of $ \{\textbf{m}_{\lambda i}\}_{\Lambda\times I} \subset J\widehat{\otimes}J $, we get a symmetric approximate diagonal for $J$. So $ J $ is symmetrically pseudo-amenable.
 \end{proof}
 We have the following result from the above theorem.
\begin{cor}\label{cor sa ideal}
Let $ \uu $ be a symmetrically pseudo-amenable Banach algebra and let $ J $ be a two-sided closed ideal of $ \uu $. If $ J $ has a bounded approximate identity, then $ J $ is symmetrically pseudo-amenable.
\end{cor}
\begin{prop}\label{union}
 Let $\uu$ be a Banach algebra with a system of closed subalgebras $\lbrace\uu_{\gamma} \, : \, \gamma\in\Gamma \rbrace$ such that 
 \begin{itemize}
 \item[(i)] $\overline{(\bigcup_{\gamma\in\Gamma} \uu_{\gamma})}=\uu$;
 \item[(ii)] if $\gamma_1 , \gamma_2 \in \Gamma$ then there is $\gamma\in\Gamma$ with $\uu_{\gamma_1}\cup\uu_{\gamma_2}\subseteq\uu_{\gamma}$;
 \item[(iii)] for each $\gamma\in \Gamma$, $\uu_{\gamma}$ is a symmetrically pseudo-amenable Banach algebra.
 \end{itemize}
 Then $\uu$ is symmetrically pseudo-amenable.
 \end{prop}
 \begin{proof}
 For an arbitrary $ \varepsilon> 0 $ and finite set $ F\subseteq  \bigcup_{\gamma\in\Gamma} \uu_{\gamma}$, by (ii) we choose $ \gamma\in\Gamma $ with $F\subseteq \uu_{\gamma}$. Since $\uu_{\gamma}$ is symmetrically pseudo-amenable, there is a symmetric element $\textbf{t}_{F,\varepsilon}\in \uu_{\gamma}\otimes\uu_{\gamma}\subseteq \uu \otimes\uu$ such that 
  \[\Vert a \textbf{t}_{F,\varepsilon}-\textbf{t}_{F,\varepsilon} a\Vert<\varepsilon \]
 and
 \[\Vert \pi(\textbf{t}_{F,\varepsilon})a-a\Vert<\varepsilon \]
 for each $ a\in F $. So the result follows from (i). 
 \end{proof}
\section{Some classes of symmetrically pseudo-amenable Banach algebras}
In the previous section we saw some classes of symmetrically pseudo-amenable Banach algebras, especially that every commutative pseudo-amenable Banach algebra is symmetrically pseudo-amenable. In this section, we introduce some other classes of symmetrically pseudo-amenable Banach algebras.  
\par 
First, we study a class of Banach algebras that belongs to the class of $\ell^{1}$-Munn algebras. Let $\mathbb{N}$ be the set of natural numbers. We denote by $M_{\mathbb{N}}(\mathbb{C})$, the set of $\mathbb{N}\times \mathbb{N}$ matrices $(a_{ij})$ with entries in $\mathbb{C}$ such that 
\[\Vert (a_{ij}) \Vert =\sum_{i,j\in \mathbb{N}}\vert a_{ij}\vert <\infty. \]
Then $M_{\mathbb{N}}(\mathbb{C})$ with the usual matrix multiplication is a Banach algebra that belongs to the class of $\ell^{1}$-Munn algebras. We write $E_{ij}$ for the matrix units and $aE_{ij}$ for the matrix with the $a$ at the $(i, j)$-entry and $0$ in all other entries. 
\begin{thm}\label{Mn}
The Banach algebra $M_{\mathbb{N}}(\mathbb{C})$ is symmetrically pseudo-amenable.
\end{thm}
\begin{proof}
For $n\in \mathbb{N}$ denote the finite set $\lbrace 1,2, \ldots n \rbrace$ by $\mathbb{N}_n$, and define 
\[ \textbf{t}_n=\dfrac{1}{n}\sum_{i,j\in \mathbb{N}_n} E_{ij}\otimes E_{ji} \in M_{\mathbb{N}}(\mathbb{C})\widehat{\otimes}M_{\mathbb{N}}(\mathbb{C}).\]
We will show that $\lbrace \textbf{t}_n\rbrace_{n\in \mathbb{N}}$ is a symmetric approximate diagonal for $M_{\mathbb{N}}(\mathbb{C})$. Let $A=(a_{ij})\in M_{\mathbb{N}}(\mathbb{C})$. According to the definition of $M_{\mathbb{N}}(\mathbb{C})$, for each $\varepsilon> 0$, there is an $n_0 \in \mathbb{N}$ such that for all $n\geq n_0$ we have $\sum_{i,j\notin \mathbb{N}_n}\vert a_{ij} \vert <\varepsilon$. For $n\in \mathbb{N}$ we have 
\[\pi(\textbf{t}_n)=\dfrac{1}{n}\sum_{i\in \mathbb{N}_n} nE_{ii}=I_n, \]
where $I_n:=\sum_{i=1}^{n} E_{ii}$.  So for $n\geq n_0$
\[\Vert \pi(\textbf{t}_n)A-A \Vert =\Vert I_n A-A \Vert\leq \sum_{i,j\notin \mathbb{N}_n}\vert a_{ij} \vert <\varepsilon . \]
Consequentially, $\pi(\textbf{t}_n)A\longrightarrow A$ for each $A\in M_{\mathbb{N}}(\mathbb{C})$. For $n\in \mathbb{N}$ define the matrices $A_r$, $1 \leq r\leq 4$ as follows
\[ A_1=(x_{ij}),~~ \text{where}~~ x_{ij}=a_{ij}~~ \text{for}~~ 1\leq i,j \leq n; ~x_{ij}=0~~ \text{for}~~ i> n~~ \text{or}~~ j> n ; \]
\[ A_2=(x_{ij}),~~ \text{where}~~ x_{ij}=a_{ij}~~ \text{for}~~ 1\leq i \leq n~~ \text{and}~~ j>n ;~ x_{ij}=0~~ \text{for}~~ i> n~~ \text{or}~~ 1\leq j \leq n; \]
\[ A_3=(x_{ij}),~~ \text{where}~~ x_{ij}=a_{ij}~~ \text{for}~~ i> n~~ \text{and}~~ 1\leq j \leq n ;~ x_{ij}=0~~ \text{for}~~ 1\leq i \leq n~~ \text{or}~~ j> n ;\]
and 
\[ A_4=(x_{ij}),~~ \text{where}~~ x_{ij}=a_{ij}~~ \text{for}~~ i> n~~ \text{and}~~ j> n ;~ x_{ij}=0~~ \text{for}~~ 1\leq i \leq n~~ \text{or}~~ 1\leq j\leq n .\]
It can be seen by routine calculations that
\[A_1 \textbf{t}_n=\textbf{t}_n A_1=\dfrac{1}{n}\sum_{k,l,j\in \mathbb{N}_n} a_{kl}E_{kj}\otimes E_{jl}; \]
\[ A_2 \textbf{t}_n=0 \quad \text{and} \quad \textbf{t}_n A_2=\dfrac{1}{n}\sum_{1\leq j , k \leq n< l} a_{kl}E_{kj}\otimes E_{jl};\]
\[  \textbf{t}_n A_3=0 \quad \text{and} \quad A_3\textbf{t}_n =\dfrac{1}{n}\sum_{1\leq j , l \leq n< k} a_{kl}E_{kj}\otimes E_{jl};\]
\[A_4 \textbf{t}_n=\textbf{t}_n A_4=0. \]
So for $n\geq n_0$
\[\Vert A\textbf{t}_n -\textbf{t}_n A \Vert = \Vert A_3\textbf{t}_n - \textbf{t}_n A_2\Vert \leq \sum_{1\leq l \leq n < k}\vert a_{ij}\vert +\sum_{1\leq k \leq n < l}\vert a_{ij}\vert \leq  \sum_{i,j\notin \mathbb{N}_n}\vert a_{ij} \vert <\varepsilon . \]
Hence $A\textbf{t}_n -\textbf{t}_n A \longrightarrow 0$. Also, for each $n\in \mathbb{N}$, it is clear that $\textbf{t}_n^{\circ}=\textbf{t}_n$. Therefore, $M_{\mathbb{N}}(\mathbb{C})$ is symmetrically pseudo-amenable.
\end{proof}
In the following, we will discuss some algebras over locally compact groups.
\begin{thm}\label{L1}
Let $G$ be a locally compact group. The group algebra $L^{1}(G)$ is symmetrically pseudo-amenable if and only if $G$ is an amenable group. 
\end{thm}
\begin{proof}
Suppose that $L^{1}(G)$ is symmetrically pseudo-amenable. So $L^{1}(G)$ is pseudo-amenable, and hence from \cite[Proposition 4.1]{ghahramani 2007} it follows that $G$ is an amenable group. Conversely, if $G$ is amenable, by \cite[Theorem 4.1]{johnson 1996} we have $L^{1}(G)$ is symmetrically pseudo-amenable.
\end{proof}
For any compact group $G$, $L^{2}(G)$ is non-amenable, except in the finite-dimensional cases. In the next theorem, we see that $L^{2}(G)$ is symmetrically pseudo-amenable.
\begin{thm}\label{L2}
For any compact group $G$, the Banach algebra $L^{2}(G)$ is symmetrically pseudo-amenable.
\end{thm}
\begin{proof}
By \cite[$\S$ 32. Theorem 1]{naimark}, the group algebra $L^{2}(G)$ ($G$ compact) is the $\ell_2$-direct sum of its minimal two-sided ideals $I_{\alpha}$, each of which is completely isomorphic to an algebra $M_n(\mathbb{C})$ ($n\in \mathbb{N}$, and a matrix for each ideal). We know that each $M_n(\mathbb{C})$ is symmetrically amenable (see Example \ref{not amenable}). Hence from Proposition \ref{directsum} it follows that $L^{2}(G)$ is symmetrically pseudo-amenable.
\end{proof}
We note that according to \cite[Proposition 2.7]{johnson 1996} every strongly amenable $C^*$-algebra is symmetrically amenable and therefore is symmetrically pseudo-amenable.
\section{Jordan derivations of symmetrically pseudo-amenable Banach algebras}
Let $\uu$ be a Banach algebra. In the following, $\uu^{\sharp}$ means the unitization of $\uu$ with the $\ell^{1}$-norm, which we consider in any case, whether $\uu$ is unital or not. The Banach algebra $\uu^{\sharp}$ is unital with unity $e$ where $\Vert e \Vert =1$. Let $X$ be a Banach $\uu$-bimodule and we turn $X$ into a Banach $\uu^{\sharp}$-bimodule by defining $1x=x1=x$ for each $x\in X$, and hence $ex=xe=x$ for each $x\in X$. Let $x\in X$. The mapping $(a,b)\mapsto axb$ from $\uu^{\sharp}\times \uu^{\sharp}$ into $X$ is bilinear and $\Vert axb \Vert\leq M_X \Vert a \Vert \Vert x \Vert \Vert b\Vert$ for all $a,b\in \uu^{\sharp}$, where $M_X=\sup \lbrace \Vert ay \Vert , \Vert ya \Vert \,  : \, a\in \uu, y \in X , \Vert a\Vert=\Vert y \Vert =1\rbrace$. Thus we can define a continuous linear operator $\psi_x : \uu^{\sharp} \otimes \uu^{\sharp}\rightarrow X$ by $\psi_x(a\otimes b)=axb$ for all $a,b \in \uu^{\sharp}$. It is clear that $\Vert \psi_x \Vert\leq M_X \Vert x\Vert$. Let $T:\uu\rightarrow X$ be a bounded linear map, and and we extend $T$ to $\uu^{\sharp}$ by putting $T(1) = 0$. So $T(e)=0$. Then $\Phi_{T}: \uu^{\sharp} \otimes \uu^{\sharp} \rightarrow X$ is the bounded linear mapping  specified by $\Phi_T(a\otimes b)=aT(b)$ with $\Vert  \Phi_T\Vert \leq\Vert T\Vert$
\par 
Now we are ready to state the main results of this section. In the following theorems, it is assumed that $\psi_{x}$ and $\Phi_{T}$ are defined as above.
\begin{thm}\label{jordan}
Let $\uu$ be a Banach algebra such that $\uu^{\sharp}$ is symmetrically pseudo-amenable with the a symmetric approximate diagonal $\lbrace \textbf{t}_{\lambda}\rbrace_{\lambda\in\Lambda}$. Suppose $X$ is a Banach $\uu$-bimodule such that 
\begin{itemize}
\item[(i)] for each $x\in X$ the net $\lbrace\psi_{x} (\textbf{t}_{\lambda})\rbrace_{\lambda\in\Lambda}$ is bounded, and
\item[(ii)] for each bounded Jordan derivation $D:\uu \rightarrow X$ the net $\lbrace\Phi_{D}(\textbf{t}_{\lambda})\rbrace_{\lambda\in\Lambda}$ is bounded.
\end{itemize}
Then every bounded Jordan derivation from $ \uu $ to $ X $ is a derivation.
\end{thm}
\begin{proof}
Suppose that $D:\uu \rightarrow X$  is a bounded Jordan derivation. We extend $D$ to $\uu^{\sharp}$ by putting $D(1) = 0$. So $D(e)=0$, where $e$ is the unity of $\uu^{\sharp}$. Then 
\[\Phi_D(a\textbf{t}_{\lambda}-\textbf{t}_{\lambda}a)=a\Phi_D(\textbf{t}_{\lambda}) +\Phi_D(a\circ \textbf{t}_{\lambda})-\pi(\textbf{t}_{\lambda})D(a)-\Phi_D(\textbf{t}_{\lambda})a-\Phi_D( \textbf{t}_{\lambda}\circ a)-\psi_{D(a)}(\textbf{t}_{\lambda}),\]
for each $a\in \uu^{\sharp}$. Let $x_{\lambda}:=\Phi_D(\textbf{t}_{\lambda})$. So
\[\pi(\textbf{t}_{\lambda})D(a)=(ax_{\lambda}-x_{\lambda}a)-\Phi_D(a\textbf{t}_{\lambda}-\textbf{t}_{\lambda}a)+\Phi_D(a\circ \textbf{t}_{\lambda} -\textbf{t}_{\lambda}\circ a) -\psi_{D(a)}(\textbf{t}_{\lambda}),\]
for each $a\in \uu^{\sharp}$. We have $\pi(\textbf{t}_{\lambda})\longrightarrow e$ and since $X$ is unital over $\uu^{\sharp}$, $\pi(\textbf{t}_{\lambda})D(a)\longrightarrow D(a)$ for $a\in \uu^{\sharp}$. On the other hand
\[\Vert \Phi_D(a\textbf{t}_{\lambda}-\textbf{t}_{\lambda}a)\Vert\leq \Vert D\Vert \Vert a\textbf{t}_{\lambda}-\textbf{t}_{\lambda}a \Vert \longrightarrow 0 \]
and 
\[\Vert \Phi_D(a\circ \textbf{t}_{\lambda} -\textbf{t}_{\lambda}\circ a)\Vert\leq \Vert D\Vert \Vert a\circ \textbf{t}_{\lambda} -\textbf{t}_{\lambda}\circ a \Vert \longrightarrow 0 \]
for $a\in \uu^{\sharp}$. Therefore, 
\begin{equation}\label{lim}
D(a)=\lim _{\lambda}((ax_{\lambda}-x_{\lambda}a)-\psi_{D(a)}(\textbf{t}_{\lambda}))
\end{equation} 
for $a\in \uu^{\sharp}$. Now, viewing $ X $  as a closed $\uu^{\sharp}$-subbimodule of $X^{**}$, and hence $ D $ is a bounded Jordan derivation from $ \uu^{\sharp} $ to $ X^{**} $. Since $\lbrace x_{\lambda}\rbrace_{\lambda\in \Lambda}$ is bounded, define $ \Omega\in X^{**} $ as follows:
\[ \langle\Omega,f\rangle=Lim_{\lambda}\langle x_{\lambda},f\rangle, \]
 where $f\in X^{*}  $ and $ Lim_{\lambda} $ is a generalized limit on $ \Lambda $. Also, by our assumption, define the bounded linear map $ \Delta:\uu^{\sharp}\rightarrow X^{**} $ by
 \[\langle\Delta(a),f\rangle=Lim_{\lambda}\langle \psi_{D(a)}(\textbf{t}_{\lambda}),f\rangle ,\]
where $a\in\uu^{\sharp}$ and $f\in X^{*} $. It follows from \eqref{lim} that 
\begin{equation*}
\begin{split}
\langle D(a),f \rangle &=Lim_{\lambda}\langle ax_{\lambda}-x_{\lambda}a, f\rangle -Lim_{\lambda}\langle \psi_{D(a)}(\textbf{t}_{\lambda}) ,f\rangle \\&
= \langle a\Omega -\Omega a , f\rangle -\langle\Delta(a),f\rangle
\end{split}
\end{equation*}
for any $a\in\uu^{\sharp}$ and $f\in X^{*} $. So
\[D(a)=a\Omega -\Omega a-\Delta(a)\]
for each $a\in\uu^{\sharp}$, and hence $ \Delta $ is a bounded Jordan derivation. For each $a,b \in \uu^{\sharp}$ we have 
\begin{equation*}
\begin{split}
\langle a\Delta (b),f  \rangle &= Lim_{\lambda}\langle \psi_{D(b)}(\textbf{t}_{\lambda}),fa\rangle\\&
= Lim_{\lambda}\langle a\psi_{D(b)}(\textbf{t}_{\lambda}),f\rangle\\&
=Lim_{\lambda}\langle \psi_{D(b)}(a\textbf{t}_{\lambda}),f\rangle\\&
=Lim_{\lambda}\langle \psi_{D(b)}(\textbf{t}_{\lambda}a),f\rangle\\&
=Lim_{\lambda}\langle \psi_{D(b)}(\textbf{t}_{\lambda})a,f\rangle\\&
=\langle \Delta (b) a,f  \rangle
\end{split}
\end{equation*}
for all $f\in X^{*} $, because $a\textbf{t}_{\lambda}-\textbf{t}_{\lambda}a\longrightarrow 0$ and the nets $\lbrace\psi_{D(b)} (a\textbf{t}_{\lambda})\rbrace_{\lambda\in\Lambda}$, $\lbrace\psi_{D(b)} (\textbf{t}_{\lambda}a)\rbrace_{\lambda\in\Lambda}$ are bounded. Thus
\begin{equation}\label{commutative}
a\Delta (b)=\Delta (b) a
\end{equation}
for each $a,b \in \uu^{\sharp}$. Now we do the same process for $ \Delta $ as we did earlier for $ D$ and therefore
\[ \Delta(a)=a\Omega_{1}-\Omega_{1}a-\Delta_{1}(a)  \]
for $a\in \uu^{\sharp}$, where $ \Omega_{1}\in X^{**} $ and $ \Delta_{1} $ is a bounded linear map from $ \uu $ to $ X^{**} $ defined by
\[  \langle\Delta_{1}(a),f\rangle=Lim_{\lambda}\langle \psi_{\Delta(a)}(\textbf{t}_{\lambda}),f\rangle \]
for $a\in\uu^{\sharp}$ and $f\in X^{*} $ (By condition (ii) of our assumption $\Delta_{1}$ is well-defined). It follows from \eqref{commutative} that 
\[\psi_{\Delta(a)}(\textbf{t})=\pi(\textbf{t})\Delta(a) \]
for each $a\in \uu^{\sharp}$ and $\textbf{t}\in  \uu^{\sharp} \otimes \uu^{\sharp}$. So $\psi_{\Delta(a)}(\textbf{t}_{\lambda})\longrightarrow \Delta(a)$. Consequentially, $\Delta_{1}(a)=\Delta(a)$ for all $a\in\uu^{\sharp}$ and 
\[ \Delta(a)=a(\dfrac{1}{2}\Omega_{1})-(\dfrac{1}{2}\Omega_{1})a \]
for $a\in\uu^{\sharp}$. According to this identity and $D(a)=a\Omega -\Omega a-\Delta(a)$ we have
 \[D(a)=a(\Omega-\dfrac{1}{2}\Omega_{1})-(\Omega-\dfrac{1}{2}\Omega_{1})a \]
for $a\in\uu^{\sharp}$, and hence  $ D $ is a derivation.
\end{proof}
We note that if $\uu$ is a commutative approximately amenable Banach algebra, then by \cite[Theorem 3.1-(iv)]{ghahramani 2007} and Corollary \ref{symmetric commutative}, $\uu^{\sharp}$ is symmetrically pseudo-amenable. Examples of this kind of Banach algebras are given in \cite{ghahramani 2008}.
\par 
We have the following result that is proved in \cite[Theorem 6.2]{johnson 1996}. Therefore, it can be said that Theorem \ref{jordan} is a generalization of \cite[Theorem 6.2]{johnson 1996}.
\begin{cor}\label{jordan johnson}
Let $\uu$ be a symmetrically amenable Banach algebra and $X$ be a Banach $\uu$-bimodule. Then every bounded Jordan derivation from $ \uu $ to $ X $ is a derivation.
\end{cor}
\begin{proof}
By \cite[Theorem 3.1]{johnson 1996}, $\uu^{\sharp}$ is symmetrically amenable, and so it has a bounded symmetric approximate diagonal $\lbrace \textbf{t}_{\lambda}\rbrace_{\lambda\in\Lambda}$. Since $\psi_x : \uu^{\sharp} \otimes \uu^{\sharp}\rightarrow X$ for each $x\in X$ and $\Phi_{T}: \uu^{\sharp} \otimes \uu^{\sharp}\rightarrow X$ for each bounded linear map $T:\uu^{\sharp}\rightarrow X$ are bounded mappings, it follows from boundness of $\lbrace \textbf{t}_{\lambda}\rbrace_{\lambda\in\Lambda}$ that the conditions (i) and (ii) of Theorem \ref{jordan} are satisfied. Therefore, the desired result is proved.
\end{proof}
In the following theorem, we consider a special type of Jordan derivations.  
\begin{thm}\label{jordan 2}
Let $\uu$ be a Banach algebra such that $\uu^{\sharp}$ is symmetrically pseudo-amenable and $X$ is a Banach $\uu$-bimodule. Then every bounded central Jordan derivation from $ \uu $ to $ X $ is a derivation.
\end{thm}
\begin{proof}
Suppose that $D:\uu \rightarrow X$  is a bounded central Jordan derivation and $\lbrace \textbf{t}_{\lambda}\rbrace_{\lambda\in\Lambda}$ is a symmetric approximate diagonal for $\uu^{\sharp}$. We extend $D$ to $\uu^{\sharp}$ by putting $D(1) = 0$. With the same process of proving Theorem \ref{jordan}, it is proved that
\[D(a)=\lim _{\lambda}((ax_{\lambda}-x_{\lambda}a)-\psi_{D(a)}(\textbf{t}_{\lambda})) \]
for $a\in \uu^{\sharp}$. Since $D$ is central, it follows that 
\[\psi_{D(a)}(\textbf{t})=\pi(\textbf{t})D(a) \]
for each $a\in \uu^{\sharp}$ and $\textbf{t}\in  \uu^{\sharp} \otimes \uu^{\sharp}$. Thus $\psi_{D(a)}(\textbf{t}_{\lambda})\longrightarrow D(a)$ and 
\[ D(a)=\dfrac{1}{2}\lim _{\lambda}(ax_{\lambda}-x_{\lambda}a)  \]
for $a\in\uu^{\sharp}$. 
\end{proof}
Let $\uu$ be a Banach algebra. The Banach $\uu$-bimodule $X$ is called \textit{symmetric} if $ax = x a$, for all $a \in \uu$ and $x \in X$. According to Theorem \ref{jordan 2}, we have the following result which checks the bounded Jordan derivations into a certain class of Banach bimodules.
\begin{cor}\label{symetric Jordan}
If $\uu$ is a Banach algebra such that $\uu^{\sharp}$ is symmetrically pseudo-amenable, then every bounded Jordan derivation from $ \uu $ to a symmetric Banach $\uu$-bimodule $ X $ is a derivation.
\end{cor}
\section{Lie derivations of symmetrically pseudo-amenable Banach algebras}
In this section, we consider $\uu^{\sharp}$ as in the previous section and convert a $\uu$-bimodule to a $\uu^{\sharp}$-bimodule. Also, $\Phi_{T}$ is defined as in the previous section. The following lemma is about central derivations.
\begin{lem}\label{central deri}
Let $\uu$ be a Banach algebra such that $\uu^{\sharp}$ is symmetrically pseudo-amenable and $X$ be a Banach $\uu$-bimodule. Then every bounded central derivation from $ \uu $ to $ X $ is a derivation.
\end{lem}
\begin{proof}
Let $\delta:\uu \rightarrow X$ be a bounded central derivation and $\lbrace \textbf{t}_{\lambda}\rbrace_{\lambda\in\Lambda}$ be a symmetric approximate diagonal for $\uu^{\sharp}$. We extend $\delta$ to $\uu^{\sharp}$ by putting $\delta(1) = 0$. From the fact that $\delta$ is central, for each $a,b,c\in\uu^{\sharp}$ we have
\begin{equation*}
\begin{split}
\Phi_{\delta}(ab\otimes c)-\Phi_{\delta}(b\otimes ca)&=ab\delta(c)-b\delta(ca)\\&
=ab\delta(c)-\delta(ca)b\\&
=ab\delta(c)-c\delta(a)b-\delta(c)ab\\&
=cb\delta(a).
\end{split}
\end{equation*}
So 
\[\Phi_{\delta}(a\textbf{t}_{\lambda}-\textbf{t}_{\lambda}a)=\pi(\textbf{t}_{\lambda}^{\circ})\delta(a) \]
for all $a\in \uu^{\sharp}$. Since $\textbf{t}_{\lambda}^{\circ}=\textbf{t}_{\lambda}$ ($\lambda\in \Lambda$), $a\textbf{t}_{\lambda}-\textbf{t}_{\lambda}a\longrightarrow 0$ and $\pi(\textbf{t}_{\lambda})\delta(a)\longrightarrow \delta(a)$, it follows that $\delta =0$.
\end{proof}
The restriction of a central derivation to a subalgebra is central so we can extend Lemma \ref{central deri} to the following.
\begin{cor}\label{central der ext}
Let $\uu$ be the smallest closed subalgebra which contains all the closed subalgebras $\mathfrak{V}$ of $\uu$ such that $\mathfrak{V}^{\sharp}$ is symmetrically pseudo-amenable. Then every bounded central derivation with domain $\uu$ is 0. 
\end{cor} 
\begin{lem}\label{maps into submodule}
Let $\uu$ be a Banach algebra such that $\uu^{\sharp}$ is symmetrically pseudo-amenable. Suppose that $Y$ is a Banach $\uu$-bimodule and $ X $ is a closed $ \uu $-subbimodule of $ Y $. If $ \delta:\uu\rightarrow Y $ is a bounded derivation and $ \tau:\uu\rightarrow \mathcal{Z}_{\uu}(Y)$ is a linear map such that $ (\delta+\tau)(\uu)\subseteq X $, then $ \delta(\uu)\subseteq X $ and $ \tau(\uu)\subseteq\mathcal{Z}_{\uu}(X) $. 
\end{lem}
 \begin{proof}
 Let $ \pi_{X}:Y\rightarrow Y/X $ be the quotient map where $ W=Y/X $ is the quotient Banach $ \uu $-bimodule. We have
\[0=\pi_{X}\circ(\delta+\tau)=\pi_{X}\circ\delta+\pi_{X}\circ\tau .\]
Hence
\[\pi_{X}\circ\delta=-\pi_{X}\circ\tau . \]
Since $ \pi_{X} $ maps $ \mathcal{Z}_{\uu}(Y) $ into $ \mathcal{Z}_{\uu}(W) $ and $ \tau(\uu)\subseteq\mathcal{Z}_{\uu}(Y) $, it follows that 
\[ \pi_{X}\circ\delta(\uu)=-\pi_{X}\circ\tau(\uu)\subseteq\mathcal{Z}_{\uu}(W). \] 
So by the fact that $ \pi_X$ is a bounded module homomorphism, $ \pi_X\circ\delta$ is a bounded central derivation from $ \uu $ into $ W $. According to Lemma \ref{central deri}, $ \pi_X\circ\delta=0 $, and hence $\delta(\uu)\subseteq X $. Now from assumption and the obtained result, we have $ \tau(\uu)\subseteq X\cap\mathcal{Z}_{\uu}(Y)=\mathcal{Z}_{\uu}(X)$.
\end{proof}
In the following theorem we  state the main result of this section.
\begin{thm}\label{Lie}
Let $\uu$ be a Banach algebra such that $\uu^{\sharp}$ is symmetrically pseudo-amenable with the a symmetric approximate diagonal $\lbrace \textbf{t}_{\lambda}\rbrace_{\lambda\in\Lambda}$. Suppose $X$ is a Banach $\uu$-bimodule such that for each $x\in X$ the net $\lbrace\psi_{x} (\textbf{t}_{\lambda})\rbrace_{\lambda\in\Lambda}$ is bounded, and $D:\uu \rightarrow X$ is a bounded Lie derivation such that the net $\lbrace\Phi_{D}(\textbf{t}_{\lambda})\rbrace_{\lambda\in\Lambda}$ is bounded. Then there exist a bounded derivation $ d:\uu\rightarrow X $ and a bounded central trace $\tau:\uu\rightarrow\mathcal{Z}_{\uu}(X)  $ such that $ D=d+\tau $.
\end{thm}
\begin{proof}
Assume that $D:\uu \rightarrow X$  is a bounded Lie derivation. We extend $D$ to $\uu^{\sharp}$ by putting $D(1) = 0$. Then 
\[\Phi_D(a\textbf{t}_{\lambda}-\textbf{t}_{\lambda}a)=a\Phi_D(\textbf{t}_{\lambda}) -\Phi_D(a\circ \textbf{t}_{\lambda})+\psi_{D(a)}(\textbf{t}_{\lambda}) -\pi(\textbf{t}_{\lambda})D(a)+\Phi_D( \textbf{t}_{\lambda}\circ a)-\Phi_D(\textbf{t}_{\lambda})a\]
for each $a\in \uu^{\sharp}$. Let $x_{\lambda}:=\Phi_D(\textbf{t}_{\lambda})$. Since $a\textbf{t}_{\lambda}-\textbf{t}_{\lambda}a\longrightarrow 0$, $a\circ\textbf{t}_{\lambda}-\textbf{t}_{\lambda}\circ a\longrightarrow 0$ and $\pi(\textbf{t}_{\lambda})D(a)\longrightarrow D(a)$, it follows that 
\begin{equation}\label{lim2}
D(a)=\lim _{\lambda}((ax_{\lambda}-x_{\lambda}a)+\psi_{D(a)}(\textbf{t}_{\lambda}))
\end{equation} 
for $a\in \uu^{\sharp}$. Now, viewing $ X $  as a closed $\uu^{\sharp}$-subbimodule of $X^{**}$, and hence $ D $ is a bounded Lie derivation from $ \uu^{\sharp} $ to $ X^{**} $. In view of our assumptions, define $ \Omega\in X^{**} $ and the bounded linear map $ \tau:\uu^{\sharp}\rightarrow X^{**} $ by
\[ \langle\Omega,f\rangle=Lim_{\lambda}\langle x_{\lambda},f\rangle \]
and
\[\langle\tau(a),f\rangle=Lim_{\lambda}\langle \psi_{D(a)}(\textbf{t}_{\lambda}),f\rangle ,\]
where $a\in\uu^{\sharp}$, $f\in X^{*} $ and $ Lim_{\lambda} $ is a generalized limit on $ \Lambda $.  It follows from \eqref{lim2} that 
\[\langle D(a),f \rangle = \langle a\Omega -\Omega a , f\rangle +\langle\tau(a),f\rangle \]
for any $a\in\uu^{\sharp}$ and $f\in X^{*} $. Consequentially,
\[D(a)=a\Omega -\Omega a+\tau(a)\]
for each $a\in\uu^{\sharp}$. The linear map $d:\uu^{\sharp}\rightarrow X^{**}  $ defined by $d(a)=a\Omega-\Omega a$ is a continuous derivation, and therefore $D=d+\tau $. Also, with a proof similar to the proof of centrality of $\Delta $ in the proof of Theorem \ref{jordan}, we have $  \tau(a)\in\mathcal{Z}_{\uu^{\sharp}}(X^{**})\subseteq \mathcal{Z}_{\uu}(X^{**})$. So $\tau=D-d $ is a bounded Lie derivation, and from the fact that $ \tau(\uu)\subseteq\mathcal{Z}_{\uu^{\sharp}}(X^{**}) $, it follows that $  \tau([a,b])=0$ for every $ a,b\in\uu^{\sharp} $.  The conditions of Lemma \ref{maps into submodule} hold for $ d$ and $\tau $ on $\uu$, hence $ d$ maps $\uu$ to $X $ and $ \tau$ maps $\uu$ to $\mathcal{Z}_{\uu}(X) $.
\end{proof}
The following result is immediate.
\begin{cor}\label{Lie all}
Let $\uu$ be a Banach algebra such that $\uu^{\sharp}$ is symmetrically pseudo-amenable with the a symmetric approximate diagonal $\lbrace \textbf{t}_{\lambda}\rbrace_{\lambda\in\Lambda}$. Suppose $X$ is a Banach $\uu$-bimodule such that 
\begin{itemize}
\item[(i)] for each $x\in X$ the net $\lbrace\psi_{x} (\textbf{t}_{\lambda})\rbrace_{\lambda\in\Lambda}$ is bounded, and
\item[(ii)] for each bounded Lie derivation $D:\uu \rightarrow X$ the net $\lbrace\Phi_{D}(\textbf{t}_{\lambda})\rbrace_{\lambda\in\Lambda}$ is bounded.
\end{itemize}
Then for every bounded Lie derivation $ D: \uu \rightarrow X $ there exist a bounded derivation $ d:\uu\rightarrow X $ and a bounded central trace $\tau:\uu\rightarrow\mathcal{Z}_{\uu}(X)  $ such that $ D=d+\tau $.
\end{cor}
Similar to the proof of Corollary \ref{jordan johnson}, the following result is obtained, which is proved in \cite[Theorem 9.2]{johnson 1996}. Therefore, it can be said that Theorem \ref{Lie} (and Corollary \ref{Lie all}) is a generalization of \cite[Theorem 9.2]{johnson 1996}.
\begin{cor}\label{lie johnson}
Let $\uu$ be a symmetrically amenable Banach algebra and $X$ be a Banach $\uu$-bimodule. Then for every bounded Lie derivation $ D: \uu \rightarrow X $ there exist a bounded derivation $ d:\uu\rightarrow X $ and a bounded central trace $\tau:\uu\rightarrow\mathcal{Z}_{\uu}(X)  $ such that $ D=d+\tau $.
\end{cor}

\subsection*{Acknowledgment}
The authors like to express their sincere thanks to the referee(s) for this paper.



\bibliographystyle{amsplain}
\bibliography{xbib}

\end{document}